\documentclass[12pt]{article}
\usepackage{epsfig}

\usepackage{amssymb}
\usepackage{amsmath}

\usepackage{graphicx}
\usepackage[usenames,dvipsnames]{color}

\newcommand\C{{\cal C}}

\newcommand\N{{\cal N}}

\renewcommand\L{{\mathscr L}}

\newtheorem{theorem}{Theorem}[section]
\newtheorem{lemma}[theorem]{Lemma}

\newtheorem{definition}[theorem]{Definition}

\newtheorem{remark}[theorem]{Remark}

\newenvironment{proof}{\noindent{\bf Proof}\hspace{0.5em}}
    { \null  \hfill $\square$ \par}

\usepackage[mathscr]{eucal}

\newcommand{\X}{\mathcal X}

\newcommand{\R}{\mathcal R}

\renewcommand{\S}{\mathcal S}
\renewcommand{\P}{\mathcal P}

\newcommand{\KK}{\mathscr K}

\newcommand{\abb}{{\cal A(\cal S)}}
\newcommand{\pbb}{{\cal P(\cal S)}}
\newcommand{\Bpi}{\mathscr B}

\newcommand{\ES}{{\mathbb S}}
\newcommand{\ET}{\mathbb T}
\newcommand{\EC}{\mathbb C}

\newcommand{\EX}{\mathbb X}
\newcommand{\EY}{\mathbb Y}

\newcommand\CX{-x+y}
\newcommand\CY{-\tau^q x+ \tau y}
\newcommand\CZ{-\tau^{2q} x+\tau^2 y+(\tau-\tau^q)}

\newcommand{\gt}{g_\ET}
\newcommand{\gc}{g_\EC}
\newcommand{\gs}{g_\ES}

\newcommand{\MM}{M}
\newcommand{\MMtau}{M}
\newcommand{\A}{\mathcal L}
\renewcommand{\k}{k}
\newcommand{\RR}{\mathcal R}

\newcommand{\gfq}{\mathbb F_q}
\newcommand{\gfqs}{\mathbb F_q'}
\newcommand{\gfqc}{\mathbb F_{q^3}}
\newcommand{\gfqcs}{\mathbb F_{q^3}'}

\newcommand{\TP}{{\mathcal T}_P}

\renewcommand{\r}{{q}}

\newcommand{\si}{\Sigma_\infty}
\newcommand{\li}{\ell_\infty}

\newcommand{\orsps}{order-$\r$-subplanes}
\newcommand{\orsls}{order-$\r$-sublines}
\newcommand{\orsp}{order-$\r$-subplane}
\newcommand{\orsl}{order-$\r$-subline}
\newcommand{\psline}{pencil-subline}
\newcommand{\dcsline}{dual-conic-subline}

\newcommand{\elltau}{\ell= [-\tau\tau^q, \tau+\tau^q,-1]}

\newcommand{\takeaway}{\backslash}
\newcommand{\st}{:}

\newcommand\PGL{{\rm PGL}}
\newcommand\GF{{\rm GF}}
\newcommand\PG{{\rm PG}}

\newcommand{\car}{E} 

\renewcommand\setminus{\backslash}
\newcommand{\Label}{\label}

\usepackage{setspace}

\begin{document}
%

%

\title{
The exterior splash in $\PG(6,q)$: Transversals}

\author{S.G. Barwick and Wen-Ai Jackson
\date{\today}
\\ School of Mathematical Sciences, University of Adelaide\\
Adelaide 5005, Australia
}

\maketitle

Corresponding Author: Dr Susan Barwick, University of Adelaide, Adelaide
5005, Australia. Phone: +61 8 8313 3983, Fax: +61 8 8313 3696, email:
susan.barwick@adelaide.edu.au

Keywords: Bruck-Bose representation, subplanes, exterior splash
\\ AMS code: 51E20

\begin{abstract}
Let $\pi$ be an order-$q$-subplane of $\PG(2,q^3)$ that is exterior to $\ell_\infty$. Then the exterior splash of $\pi$ is the set of $q^2+q+1$ points on $\ell_\infty$ that lie on an extended line of $\pi$. Exterior splashes are projectively equivalent to scattered linear sets of rank 3, covers of the circle geometry $CG(3,q)$, and hyper-reguli in $\PG(5,q)$. In this article we use the Bruck-Bose representation in $\PG(6,q)$ to investigate the structure of $\pi$, and 
the interaction between $\pi$ and its exterior splash. In $\PG(6,q)$, an exterior splash $\mathbb S$ has two sets of cover planes (which are hyper-reguli) and we show that each set has three unique transversals lines in the cubic extension  $\PG(6,q^3)$. These transversal lines are used to characterise the carriers of $\mathbb S$, and to characterise the sublines of $\mathbb S$.
\end{abstract}

\section{Introduction}

In \cite{barw12,BJ-iff}, the authors study \orsps\ of $\PG(2,q^3)$ and determine their representation in the Bruck-Bose representation in $\PG(6,q)$. A full characterisation in $\PG(6,q)$ was given for \orsps\ that are secant or tangent to $\li$ in $\PG(2,q^3)$. However, an \orsp\ of $\PG(2,q^3)$ that is exterior to $\li$ forms a complex structure in $\PG(6,q)$, and one of our original motivations was to investigate this structure in $\PG(6,q)$. As part of this investigation, we studied the splash of an \orsp, and found this to be a structure rich in detail, and related to many other structures in geometry, see \cite{BJ-tgt1,BJ-tgt2,BJ-ext1,lavr14}. In this article we study the structure of an exterior \orsp\ in $\PG(6,q)$, as well as the interplay between this structure and the associated exterior splash. 

Let $\pi$ be 
a subplane of $\PG(2,q^3)$ of order $q$ that is exterior to $\li$. The lines of $\pi$ meet $\li$ in a set $\ES$ of size $q^2+q+1$, called the {\em exterior splash} of $\pi$. 
Properties of the exterior splash are studied in \cite{BJ-ext1}. 
In particular, we showed that the sets of points in an exterior splash has arisen in many different situations, namely scattered linear sets of rank 3, covers of the circle geometry $CG(3,q)$, hyper-reguli in $\PG(5,q)$, and  Sherk surfaces of size $q^2+q+1$.  

This article proceeds as follows. 
In Section~\ref{sect:intro} we
introduce the notation we use for 
 the Bruck-Bose representation of $\PG(2,q^3)$ in $\PG(6,q)$, as well as presenting some other preliminary results. 
 In Section~\ref{mapsection-Bpi2} we  choose an \orsp\ $\Bpi$ in $\PG(2,q^3)$ that is exterior to $\li$, and calculate its points 
 and exterior splash. This \orsp\ will be used in many of the proofs in this article. 
In  Section~\ref{sec:sub-str}, we study the 
structure of an \orsp\ in $\PG(6,q)$.  In 
particular, we show that it contains $q^2+q+1$ twisted cubics; it is the intersection of nine quadrics; and has an interesting tangent plane at each point. 

 We next study the exterior splash $\ES$ of $\li$ in the Bruck-Bose representation in $\PG(5,q)$. By results of Bruck~\cite{bruc73b}, $\ES$ has two switching sets denoted $\EC,\ET$, which we call covers of $\ES$. The three sets $\ES,\EC,\ET$ are called hyper-reguli in \cite{ostrom}.
  In Section~\ref{sec:coord-cover}, we find coordinates for the sets $\ES, \EC,\ET$.
   In 
 Section \ref{sec:coords-covers-new}, we 
 show that each of the sets $\ES, \EC,\ET$ has a unique triple of conjugate transversal lines in the cubic extension $\PG(5,q^3)$. 
These nine transversal lines are used to characterise the carriers of an exterior splash. 
Further, these transversal lines are related to 
the set of $q-1$ disjoint splashes of $\li$ that have common carriers, and their related covers. We interpret this result in terms of replacing hyper-reguli to create Andr\'e planes.
  In Section~\ref{sec:orsl-BB} we 
  use the transversal lines to characterise the \orsls\ of an exterior splash  in terms of how the corresponding 2-reguli meet the cover planes.

\section{The Bruck-Bose representation}\Label{sect:intro}

\subsection{The Bruck-Bose representation of $\PG(2,q^3)$ in $\PG(6,q)$}\Label{BBintro}

We  introduce the
notation we will use for the 
Bruck-Bose representation of $\PG(2,q^3)$ in $\PG(6,q)$. 
We work with the finite field $\gfq=\GF(q)$, and let $\gfqs=\gfq\setminus\{0\}$.

A 2-{\em spread} of $\PG(5,\r)$ is a set of $\r^3+1$ planes that partition
$\PG(5,\r)$. 
A 2-{\em regulus} of $\PG(5,\r)$ is a
set of $\r+1$ mutually disjoint planes $\pi_1,\ldots,\pi_{\r+1}$ with
the property that if a line meets three of the planes, then it meets all
${\r+1}$ of them. A $2$-regulus $\R$ has a set of $q^2+q+1$ mutually disjoint {\em ruling lines} that meet every plane of $\R$. A $2$-regulus is uniquely determined by three mutually disjoint planes, or four  (ruling) lines (mutually disjoint and lying in general position). A $2$-spread $\S$ is {\em regular} if for any three planes in $\S$, the
$2$-regulus containing them is contained in $\S$. 
See \cite{hirs91} for
more information on $2$-spreads.

The following construction of a regular $2$-spread of $\PG(5,\r)$ will be
needed. Embed $\PG(5,\r)$ in $\PG(5,\r^3)$ and let $g$ be a line of
$\PG(5,\r^3)$ disjoint from $\PG(5,\r)$. Let $g^\r$, $g^{\r^2}$ be the
conjugate lines of $g$; both of these are disjoint from $\PG(5,\r)$. Let $P_i$ be
a point on $g$; then the plane $\langle P_i,P_i^\r,P_i^{\r^2}\rangle$ meets
$\PG(5,\r)$ in a plane. As $P_i$ ranges over all the points of  $g$, we get
$\r^3+1$ planes of $\PG(5,\r)$ that partition $\PG(5,\r)$. These planes form a
regular 2-spread $\S$ of $\PG(5,\r)$. The lines $g$, $g^\r$, $g^{\r^2}$ are called the (conjugate
skew) {\em transversal lines} of the 2-spread $\S$. Conversely, given a regular $2$-spread
in $\PG(5,\r)$,
there is a unique set of three (conjugate skew) transversal lines in $\PG(5,\r^3)$ that generate
$\S$ in this way.

We will use the linear representation of a finite
translation plane $\P$ of dimension at most three over its kernel,
due independently to
Andr\'{e}~\cite{andr54} and Bruck and Bose
\cite{bruc64,bruc66}. 
Let $\si$ be a hyperplane of $\PG(6,\r)$ and let $\S$ be a $2$-spread
of $\si$. We use the phrase {\em a subspace of $\PG(6,\r)\takeaway\si$} to
  mean a subspace of $\PG(6,\r)$ that is not contained in $\si$.  Consider the following incidence
structure:
the \emph{points} of $\abb$ are the points of $\PG(6,\r)\takeaway\si$; the \emph{lines} of $\abb$ are the $3$-spaces of $\PG(6,\r)\takeaway\si$ that contain
  an element of $\S$; and \emph{incidence} in $\abb$ is induced by incidence in
  $\PG(6,\r)$.
Then the incidence structure $\abb$ is an affine plane of order $\r^3$. We
can complete $\abb$ to a projective plane $\pbb$; the points on the line at
infinity $\li$ have a natural correspondence to the elements of the $2$-spread $\S$.
The projective plane $\pbb$ is the Desarguesian plane $\PG(2,\r^3)$ if and
only if $\S$ is a regular $2$-spread of $\si\cong\PG(5,\r)$ (see \cite{bruc69}).
For the remainder of this article, we use $\S$ to denote a regular $2$-spread of $\si\cong\PG(5,q)$.

We use the following notation.  If $T$
is a point of $\li$ in $\PG(2,q^3)$, we use $[T]$ to refer to the plane of $\S$ corresponding to $T$. More generally, if $X$
is a set of points of $\PG(2,q^3)$, then we let $[X]$ denote the
corresponding set  in $\PG(6,q)$. If $P$ is an affine point of $\PG(2,q^3)$,  we generally simplify the notation and also
use $P$  to refer to the corresponding affine point in $\PG(6,q)$, although in  some cases, to avoid confusion, we use $[P]$.

When $\S$ is a regular $2$-spread, 
we can relate the coordinates of $\pbb\cong\PG(2,\r^3)$ and $\PG(6,\r)$ as
follows. Let $\tau$ be a primitive element in $\gfqc$ with primitive
polynomial $x^3-t_2x^2-t_1x-t_0.$
Every element $\alpha\in\gfqc$
can be uniquely written as $\alpha=a_0+a_1\tau+a_2\tau^2$ with
$a_0,a_1,a_2\in\gfq$. Points in $\PG(2,\r^3)$ have homogeneous coordinates
$(x,y,z)$ with $x,y,z\in\gfqc$, not all zero. Let the line at infinity $\li$ have
equation $z=0$; so the affine points of $\PG(2,\r^3)$ have coordinates
$(x,y,1)$. Points in $\PG(6,\r)$ have homogeneous coordinates
$(x_0,x_1,x_2,y_0,y_1,y_2,z)$ with $x_0,x_1,x_2,y_0,y_1,y_2,z\in\gfq$.  Let $\si$ have equation $z=0$. 
Let $P=(\alpha,\beta,1)$ be a point of $\PG(2,\r^3)$. We can write $\alpha=a_0+a_1\tau+a_2\tau^2$ and
$\beta=b_0+b_1\tau+b_2\tau^2$ with $a_0,a_1,a_2,b_0,b_1,b_2\in\gfq$. 
We want to map the element $\alpha$ of $\gfqc$ to the vector $(a_0,a_1,a_2)$, and we use the following notation to do this:
$$[\alpha]=(a_0,a_1,a_2).$$
This gives us some notation for the Bruck-Bose map, denoted $\epsilon$, from an affine point $P=(\alpha,\beta,1)\in\PG(2,\r^3)\takeaway\li$ to the corresponding affine point $[P]\in\PG(6,\r)\takeaway\Sigma_\infty$, namely  $$\epsilon(\alpha,\beta,1)=[(\alpha,\beta,1)]=([\alpha],[\beta],1)=(a_0,a_1,a_2,b_0,b_1,b_2,1).$$
More generally, if $z\in\gfq$, then 
$\epsilon(\alpha,\beta,z)=([\alpha],[\beta],z)=(a_0,a_1,a_2,b_0,b_1,b_2,z).$

Consider the case when $z=0$, a point on $\li$ in $\PG(2,q^3)$ has coordinates
 $L=(\alpha,\beta,0)$ for some $\alpha,\beta\in\gfqc$.
 In $\PG(6,q)$, the point 
 $\epsilon(\alpha,\beta,0)=([\alpha],[\beta],0)$  is one point in the spread element $[L]$
corresponding to $L$. Moreover, the spread element $[L]$ consists of all the points $\{([\alpha x],[\beta x],0)\st x\in\gfqcs\}$. Hence the regular 2-spread $\S$ consists of the planes $\{[kx],[x],0]\st x\in\gfqc'\}$ for $k\in\gfqc\cup\{\infty\}$.

With this coordinatisation for the Bruck-Bose map, we can calculate the coordinates  of the transversal lines  of the regular $2$-spread $\S$. 
 
 \begin{lemma}\Label{S-coord-g}{\rm \cite{barw12}}
 Let $p_0=
t_1+t_2\tau-\tau^2=-\tau^q\tau^{q^2}$, $p_1=t_2-\tau=\tau^q+\tau^{q^2}$, $p_2=-1$, and $A=(p_0,p_1,p_2)$. Then in the cubic extension $\PG(6,q^3)$, one transversal line of the regular $2$-spread $\S$ contains the two points $A_1=(p_0,p_1,p_2,0,0,0,0)=(A,[0],0)$ and $A_2=(0,0,0,p_0,p_1,p_2,0)=([0],A,0)$.
\end{lemma}


%
%

\subsection{Some useful homographies}\Label{sec:eg:Mk}

In order to simplify the notation in some of the following coordinate based proofs, we define some homographies which will be useful.
 We can represent an element $x=x_0+x_1\tau+x_2\tau^2\in\gfqc$  as a point $[x]=(x_0,x_1,x_2)$ in $\PG(2,q)$. 
For $\k\in\gfqcs$, consider the homography $\zeta_\k $ in $\PGL(3,q)$ with matrix $\MM_\k $ that maps $[x]$ to $[\k x]$.
Let $k\in\gfqcs$, write $\k=\k_0+\k_1\tau+\k_2\tau^2$, then $\MM_\k=\k_0\MM_1+\k_1\MM_\tau+\k_2\MM_{\tau^2}$, and hence
\begin{eqnarray}\label{JA-eqn}
\MM_\k A=\k A&\quad{\rm and}\quad&\MM_\k A^{q^2}=\k^{q^2}A^{q^2},
\end{eqnarray}
where $A=(p_0,p_1,p_2)^t$ is defined in Lemma~\ref{S-coord-g}.
We use $\zeta_\k $ to define the homography $\theta_\k $ of $\PG(5,q)$, $\k\in\gfqc$:
\begin{eqnarray*}
\theta_\k \colon ([x],[y])\rightarrow([\k x],[y])=(\MM_\k [x],[y]).
\end{eqnarray*}
From the matrix $\MM_\tau$, we construct three more homographies of $\PG(2,q)$ with matrices $U_0,U_1,U_2$ that help with the notation in the proof of Theorem~\ref{psline-conics-special}. 
For $i=0,1,2$, (with $p_i$ as in Lemma~\ref{S-coord-g}), let
    $$U_i=(p_0I+p_1\MM_\tau +p_2\MM_\tau ^2)^{q^i}
    =
\begin{pmatrix}
p_0^{q^i}&\tau^{q^i} p_0^{q^i}&\tau^{2q^i}p_0^{q^i}\\
p_1^{q^i}&\tau^{q^i} p_1^{q^i}&\tau^{2q^i}p_1^{q^i}\\
p_2^{q^i}&\tau^{q^i} p_2^{q^i}&\tau^{2q^i}p_2^{q^i}
\end{pmatrix}.$$
Then 
$$
U_i\begin{pmatrix}a_0\\a_1\\a_2\end{pmatrix}=(a_0+a_1\tau^{q^i}+a_2\tau^{2q^i})\begin{pmatrix}p_0^{q^i}\\p_1^{q^i}\\p_2^{q^i}\end{pmatrix},  a_0,a_1,a_2\in\gfqc.$$
 Note that if $a_0,a_1,a_2\in\gfq$, and $\alpha=a_0+a_1\tau+ a_2\tau^2$, then $[\alpha]=(a_0,a_1,a_2)^t$, and we write the matrix equation as $
U_i[\alpha]=\alpha^{q^i}A^{q^i}$.
%
%
%

\subsection{Sublines in the Bruck-Bose representation}\Label{sec:sublinesinBB}

An \emph{\orsp} of $\PG(2,\r^3)$ is a subplane of
$\PG(2,q^3)$ of order $q$. Equivalently, it is an image of $\PG(2,\r)$ under $\PGL(3,\r^3)$. An \emph{\orsl} of
$\PG(2,q^3)$ is a line of an \orsp\ of $\PG(2,q^3)$. An \emph{\orsl}\ of $\PG(1,q^3)$ is defined to be  
one of the images of $\PG(1,\r)=\{(a,1)\st a\in\gfq\}\cup\{(1,0)\}$
under $\PGL(2,\r^3)$.  

In \cite{barw12,BJ-iff}, the authors determine the representation of \orsps\  and
\orsls\ of $\PG(2,q^3)$ in the Bruck-Bose representation in $\PG(6,q)$, and
we quote the results for \orsls\ which are needed in this article.  We first introduce some terminology to simplify the statements. 
Recall that $\S$ is a regular $2$-spread  in the hyperplane at infinity $\si$ in $\PG(6,q)$. 

\begin{definition}\Label{def:S-special}
\begin{enumerate}
\item 
An {\em
   $\S$-special conic}  is a non-degenerate conic $\C$ contained in a plane of $\S$, such
that 
the extension of $\C$ to $\PG(6,q^3)$
 meets the transversals of
 $\S$. 
 \item An {\em $\S$-special twisted cubic} is a
twisted cubic  $\N$ in a $3$-space
 of $\PG(6,q)\backslash\si$ about a plane of $\S$, such that the extension of $\N$ to $\PG(6,q^3)$ meets the transversals of $\S$. 
\end{enumerate}
\end{definition}

\begin{theorem}{\rm \cite{barw12}}\Label{sublinesinBB}
Let $b$ be an \orsl\ of $\PG(2,q^3)$. 
\begin{enumerate}
\item\Label{subline-secant-linfty} If $b\subset\li$, then in $\PG(6,q)$, $b$ corresponds to a 
$2$-regulus of $\S$. Conversely every $2$-regulus of $\S$ corresponds to
an \orsl\ of $\li$. 
\item If $b$ meets $\li$ in a point, then in $\PG(6,q)$, $b$  corresponds to
  a line of $\PG(6,q)\backslash\si$. Conversely every line of
  $\PG(6,q)\backslash\si$ corresponds to an \orsl\ of $\PG(2,q^3)$ tangent
  to $\li$.
\item\Label{FFA-orsl}
If $b$ is disjoint from $\li$, then in
$\PG(6,q)$, $b$ corresponds to an $\S$-special twisted cubic. Further, a
twisted cubic $\N$ of $\PG(6,q)$ corresponds to an \orsl\ of
$\PG(2,q^3)$ if and only if $\N$ is $\S$-special.
\end{enumerate}
\end{theorem}

The article 
 \cite{barw12} also determines the representation of  secant and tangent  \orsps\ of $\PG(2,q^3)$ in $\PG(6,q)$. The representation of an exterior \orsp\ in $\PG(6,q)$ is more complex to describe. One of the motivations of this work was to investigate this representation in more detail. Some aspects of the representation are discussed in more detail in Section~\ref{sec:sub-str}.

\subsection{Properties of exterior splashes}

We need the following group theoretic results about \orsps\ and exterior splashes which appear in \cite{BJ-ext1}.

\begin{theorem}\Label{transsplashes}
Consider the collineation group  $G=\PGL(3,\r^3)$ acting on $\PG(2,q^3)$. The 
subgroup $G_\ell$ fixing a line $\ell$ is transitive on the \orsps\ that are exterior to $\ell$, and is transitive on the
exterior splashes of $\ell$.
\end{theorem}

\begin{theorem}\Label{propthm}
The group $K=\PGL(3,q^3)_\pi$ acting on $\PG(2,q^3)$ and
 fixing an \orsp\ $\pi$ is transitive on  the points of $\pi$.
 \end{theorem}
 
 Consider the group  $G=\PGL(3,q^3)$ acting on $\PG(2,q^3)$. It has an important subgroup
$I=G_{\pi,\ell}$ which 
fixes an \orsp\ $\pi$, and a line $\ell$ exterior to $\pi$. 
By \cite{BJ-ext1}, $I$ fixes exactly three lines: $\ell$, and its conjugates $m$, $n$ with respect to $\pi$; and  
$I$ \ fixes exactly three points: $\car_1=\ell\cap m$, 
$\car_2=\ell\cap n$, 
$\car_3=m \cap n$, which are conjugate with respect to $\pi$.
This group $I$ identifies  two special points $\car_1=\ell\cap m$, 
$\car_2=\ell\cap n$ on $\ell$ which are called the {\em carriers} of  the exterior splash $\ES$ of $\pi$. This is consistent with the definition of carriers of a circle geometry $CG(3,q)$, see \cite{BJ-ext1}. Further, the fixed points and fixed lines of $I=G_{\pi,\ell}$ are used to define an important class of conics in an \orsp\ $\pi$ with respect to an exterior line $\ell$.

 \begin{definition} Let $\pi$ be an \orsp\ with exterior line $\ell$, and let $I=\PGL(3,q^3)_{\pi,\ell}$ have fixed lines $\ell,m,n$ and fixed points $E_1,E_2,E_3$. 
 \begin{enumerate}
 \item A conic of $\pi$ whose extension to $\PG(2,q^3)$ contains the three points $\car_1,\car_2,\car_3$ is called a $(\pi,\ell)$-{\em special conic} of $\pi$.
\item A dual conic of $\pi$  whose extension to $\PG(2,q^3)$ contains the three  lines $\ell,m,n$ is called a 
$(\pi,\ell)$-{\em special dual conic}.
\end{enumerate}
\end{definition}
Note that  special conics and special dual conics are necessarily
 irreducible (see \cite{BJ-iff}).
This definition of $(\pi,\ell)$-special conics is consistent with the definition of $\S$-special conics given in Definition~\ref{def:S-special}. 
That is, if $\C$ is an $\S$-special conic of a spread element $\alpha\in\S$, and $P=\alpha\cap g$, where $g$ is a transversal of the regular $2$-spread $\S$, then $\C$ is an $(\alpha,PP^q)$-special conic of $\alpha$.

\section{Coordinatising an exterior \orsp}\Label{mapsection-Bpi2}

Recall from Theorem~\ref{transsplashes} that the group of homographies of $\PG(2,q^3)$ is transitive on pairs $(\pi,\ell)$ where $\pi$ is an \orsp\ exterior to the line $\ell$. So if we want to use coordinates to prove a result about exterior \orsps, we can without loss of generality prove it for a particular exterior \orsp.
In this section we calculate the coordinates for an \orsp\ $\Bpi$ of $\PG(2,q^3)$ that is exterior to $\li$.  We will use this \orsp\ to analyse the exterior splash in the Bruck-Bose representation in $\PG(6,q)$.

We begin with 
the \orsp\ $\pi_0=\PG(2,q)$ of $\PG(2,q^3)$. Recall that $\li$ has equation $z=0$, so $\pi_0$  is secant to $\li$. We give a homography that maps $\pi_0$  to an \orsp\ exterior to $\li$. 
We define the following matrices:
\begin{eqnarray}
K=\begin{pmatrix}-\tau&1&0\\ -\tau^q&1&0\\ \tau\tau^q&-\tau-\tau^q&1\end{pmatrix},\quad\quad \quad
K'=\begin{pmatrix}-1&1&0\\ -\tau^q&\tau&0\\ -\tau^{2q}&\tau^2& \tau-\tau^q \end{pmatrix}. 
\label{matrixK}\end{eqnarray}
Let  $\sigma$ be the homography  of $\PG(2,q^3)$  with matrix $K$. 
Note that
as $KK'$ is a $\gfqc$-multiple of the identity matrix, it follows that $K'$ is a matrix for the inverse homography $\sigma^{-1}$.  Thus, if we write the points $X$ of $\PG(2,q^3)$ as column vectors, and the lines $\ell$ of $\PG(2,q^3)$ as row vectors, then $\sigma\colon X\mapsto KX$ and  $\sigma\colon \ell\mapsto\ell K'$.

\begin{theorem}\Label{phi-prop} In $\PG(2,q^3)$, let $\pi_0=\PG(2,q)$, and let $\Bpi=\sigma(\pi_0)$, then
\begin{enumerate}
\item \Label {phi-ext} $\Bpi$ is an \orsp\ exterior to $\li$.
\item \Label {B2-splash} The exterior splash of $\Bpi$ on $\li$ is 
\[\ES=\{ (k,1,0)\st k\in\gfqc, 
k^{q^2+q+1}=1\}\equiv\{(\tau^{(q-1)i},1,0)\st 0\le i < q^2+q+1\}.
\]
\item \Label{B2-carriers}  The carriers of the exterior splash of $\Bpi$  are $\car_1=(1,0,0)$ and $\car_2=(0,1,0)$.
\end{enumerate}
\end{theorem}

\begin{proof} 
Note that $\sigma$ maps $\pi_0=\PG(2,q)$ to $\Bpi$ and maps the line $\elltau$ to $\li=[0,0,1]$. 
By \cite[Lemma 2.4]{BJ-ext1}, $\pi_0$ is exterior to $\ell$
and has carriers $\car=(1,\tau,\tau^2)$ and $\car^q=(1,\tau^q,\tau^{2q})$ on $\ell$. 
 Hence $\Bpi$ is exterior to $\li$ and has carriers $\sigma(\car)=(0,1,0)$ and $\sigma(\car^q)=(1,0,0)$ on $\li$. 

By considering the action of $\sigma$ on the 
lines $[l,m,n]$ ($l,m,n\in\gfq$, not all zero) of $\pi_0$, we calculate the 
 lines of $\Bpi$ are  $\ell_{l,m,n}=[-l-\tau^qm-\tau^{2q}n,\ l+\tau m+\tau^2n,\ n(\tau-\tau^q)]$,   $l,m,n\in\gfq$, not all zero.
The exterior splash of $\Bpi$ consists of the points $Q_{l,m,n}=\ell_{l,m,n}\cap\li=(l+\tau m +\tau^2 n,\ (l+\tau m +\tau^2n )^q,0)$. 
Writing $y=l+\tau m +\tau^2 n,$ gives $Q_{l,m,n}\equiv(y,y^q,0)\equiv (y^{1-q},1,0)$ and writing $y=\tau^{-j}$ for some $j\in\{0,\ldots,q^3-2\}$ yields $Q_{l,m,n}\equiv(\tau^{j(q-1)},1,0)$.  Note that if we write $j=n(q^2+q+1)+i$ where $0\le i< q^2+q+1$, then $\tau^{j(q-1)}=\tau^{i(q-1)}$. So we may assume that  $Q_{l,m,n}=(\tau^{i(q-1)},1,0)$ with $0\le i< q^2+q+1$. Finally, note that the solutions to $k^{q^2+q+1}=1$ are  $\tau^{i(q-1)}$, 
$0\le i< q^2+q+1$. 
\end{proof}

\section{The structure of the subplane in $\PG(6,q)$}\Label{sec:sub-str}

If $\pi$ is an exterior \orsp\ of $\PG(2,q^3)$, then in the Bruck-Bose representation in $\PG(6,q)$,  $\pi$ corresponds to a set of  $q^2+q+1$ affine points denoted $[\pi]$. It is difficult to characterise the structure of $[\pi]$.
We note that as $\pi$ contains $q^2+q+1$ \orsls\ that are exterior to $\li$, then by Theorem~\ref{sublinesinBB}, $[\pi]$ contains $q^2+q+1$ $\S$-special twisted cubics, each lying in a $3$-space
 through a distinct plane of the exterior splash of $\pi$.
In this section we aim to determine more about the structure of $[\pi]$.

\subsection{An exterior splash is the intersection of nine quadrics}

We show that the structure $[\pi]$ of $\PG(6,q)$ corresponding to an exterior \orsp\ $\pi$ of $\PG(2,q^3)$ is the intersection of nine quadrics in $\PG(6,q)$. This is analogous to \cite[Theorem 9.2]{BJ-tgt1} which shows that a {\em tangent} \orsp\ of $\PG(2,q^3)$  corresponds to a structure in $\PG(6,q)$ that is the intersection of nine quadrics.

\begin{theorem}\Label{nine-quadrics}
 Let $\pi$ be an exterior \orsp\ in $\PG(2,q^3)$. Then in $\PG(6,q)$, the corresponding set $[\pi]$ is the intersection of nine quadrics.
\end{theorem}

\begin{proof} By Theorem~\ref{transsplashes}, we can without loss of generality prove this for the \orsp\ $\Bpi$ coordinatised in Section~\ref{mapsection-Bpi2}.
We use the homographies $\sigma, \sigma^{-1}$ with matrices $K,K'$ respectively, given in (\ref{matrixK}).
A point $P=(x,y,1)\in\PG(2,q^3)$ belongs to  $\Bpi$ if its pre-image $K'P= (\CX,\ \CY,\ \CZ)$ belongs to $\pi_0=\PG(2,q)$.  Suppose firstly that $\CX\ne 0$, then
\[
K' P\equiv\left(1, \ \frac{\CY}{\CX},\ \frac{\CZ}{\CX}\right).
\]
This belongs to $\pi_0=\PG(2,q)$ if and only if the second and third coordinates belong to $\gfq$, that is,
\begin{eqnarray}
\left(\frac{\CY}{\CX}\right)^q&=&\frac{\CY}{\CX}\label{alt-first},\\
\left(\frac{\CZ}{\CX}\right)^q&=&\frac{\CZ}{\CX}\label{alt-second}.
\end{eqnarray}
Writing $x=x_0+x_1\tau+x_2\tau^2$ and $y=y_0+y_1\tau+y_2\tau^2$, where $x_i,y_i\in\gfq$, $i=1,2,3$, then equating powers of $1,\tau,\tau^2$, yields three  quadratic equations from each condition, a total of six,  each of which represents a quadric in $\PG(6,q)$.

Secondly, suppose $\CY\ne 0$, then 
\[
K' P\equiv\left(\frac{\CX}{\CY},\ 1,\ \frac{\CZ}{\CY}\right).
\]
As before, this lies in $\pi_0$ 
if and only if 
\begin{eqnarray}
\left(\frac{\CX}{\CY}\right)^q&=&\frac{\CX}{\CY}\label{alt-third},\\
\left(\frac{\CZ}{\CY}\right)^q&=&\frac{\CZ}{\CY}\label{alt-fourth},
\end{eqnarray}
leading to a further six quadrics in $\PG(6,q)$.
The equations (\ref{alt-first}) and (\ref{alt-third}) give the same triple of quadrics. Hence 
the point $P$ lies in $\Bpi$ if and only if the point $[P]$ lies on 
 a total of nine quadrics in $\PG(6,q)$.
Finally, note that if both $\CX=0$ and $\CY=0$, then $x=y=0$ and the point $P$ has coordinates $(0,0,1)$. This satisfies all the quadratic equations  from (\ref{alt-first}), (\ref{alt-second}), (\ref{alt-fourth}), and so in $\PG(6,q)$,  $[P]$ lies on each of the nine quadrics.
\end{proof}

\subsection{Tangent planes at points of an exterior subplane}
\Label{sec:def-tgt-pl}

We now consider a point $P$ lying in an  exterior \orsp\ $\pi$ of $\PG(2,q^3)$. 
In the Bruck-Bose representation in $\PG(6,q)$, $P$ corresponds to an affine point which we   also denote by $P$. We show that in $\PG(6,q)$,
there is a unique {\em tangent plane} $\TP$ at $P$ to the structure $[\pi]$. We show that there are two equivalent  ways to define this tangent plane. Recall from Theorem~\ref{sublinesinBB} that the \orsls\ of $\pi$ correspond to twisted cubics in $\PG(6,q)$. Theorem~\ref{tgtpleqn} shows that we can define $\TP$ by looking at the tangent lines at $P$ to these twisted cubics. Then Theorem~\ref{tgt-space} shows that we can define $\TP$ by looking at the tangent space of $P$ with respect to the nine quadrics defined by $[\pi]$.

\begin{theorem}\Label{tgtpleqn}
Let $\pi$ be an exterior \orsp\ of $\PG(2,q^3)$, and let $P$ be a point of $\pi$. Label the lines of $\pi$ through $P$ by $\ell_0,\ldots,\ell_q$. In $\PG(6,q)$, $\ell_i$ corresponds to a twisted cubic $[\ell_i]$. Let $m_i$ be the unique tangent line to $[\ell_i]$ through $P$. Then the lines $m_0,\ldots,m_q$ lie in a plane $\TP$, called the {\sl tangent plane} of $[\pi]$ at $P$.
\end{theorem}

\begin{proof} By Theorems~\ref{transsplashes} and \ref{propthm}, we can without loss of generality prove this for the \orsp\ $\Bpi$ coordinatised in Section~\ref{mapsection-Bpi2}, and the point $P=(0,0,1)$ of $\Bpi$. First consider the \orsp\ $\pi_0=\PG(2,q)$. The point 
$P=(0,0,1)$ lies in $\pi_0$, and the lines of $\pi_0$ through $P$ have coordinates $\ell_m'=[m,1,0]$, $m\in\gfq\cup\{\infty\}$. Points on the line $\ell_m'$ distinct from $P$ have coordinates $P_x'=(1,-m,x)$ for $x\in\gfq$.
We map the plane $\pi_0$ to $\Bpi$ using the homography $\sigma$ with matrix $K$ given in (\ref{matrixK}).
As $\sigma(P)=P$, the lines of $\Bpi$ through $P$ are $\ell_m=\sigma(\ell_m')$, $m\in\gfq\cup\{\infty\}$. Points on the line $\ell_m$ distinct from $P$ have coordinates $P_x=\sigma(P_x')=(-\tau-m,\ -\tau^q-m,\ \tau\tau^q+(\tau+\tau^q)m+x)$, for $x\in\gfq$.

To convert this to a coordinate in $\PG(6,q)$, we need to multiply by an element of $\gfqc$ so that the last coordinate lies in $\gfq$. Let $F(x)=\tau\tau^q+(\tau+\tau^q)m+x$ (the third coordinate in $P_x$). As $F(x)\in\gfqc$, we have $F(x)^{q^2+q+1}\in\gfq$. So in $\PG(6,q)$, we have the point 
$
P_x= ([-(\tau+m)F(x)^{q^2+q}],\ [-(\tau^q+m)F(x)^{q^2+q}],\ F(x)^{q^2+q+1})
$.

By Theorem~\ref{sublinesinBB}, the line $\ell_m$ of $\PG(2,q^3)$ corresponds to a 
 twisted cubic $[\ell_m]=\{P_x\st x\in\gfq\}\cup\{P\}$ of $\PG(6,q)$. Consider the 
 unique tangent to $[\ell_m]$ through $P$, and let $I_m$ be the intersection of this tangent with $\si$. We will show that the points $I_m$, $m\in\gfq\cup\{\infty\}$  form a line.
To find the point of intersection $I_m$, we let $Q_x=PP_x\cap\Sigma_\infty$, then let $x\rightarrow\infty$ and calculate $I_m=Q_\infty$.
\begin{eqnarray*}
I_m&=&\lim_{x\rightarrow\infty} PP_x\cap\Sigma_\infty
=\lim_{x\rightarrow\infty} ([-(\tau+m)F(x)^{q^2+q}],[-(\tau^q+m)F(x)^{q^2+q}],0)\\
&=&([-(\tau+m)],\ -[\tau^q+m], 0).
\end{eqnarray*}
Hence the points $I_m$, $m\in\gfq\cup\{\infty\}$ form a line $\ell=\langle
([1],[1],0),([\tau],[\tau^q],0)\rangle$ in $\si$.
 Hence the tangent lines $m_0,\ldots,m_q$ to the twisted cubics of $[\pi]$ through $P$ form a plane $\TP=\langle \ell, P\rangle$ through $P$, as required.
\end{proof}

\begin{theorem}\Label{tgt-space}
Let $\pi$ be an exterior \orsp\ of $\PG(2,q^3)$, and 
let $P$ be a point of $\pi$. In $\PG(6,q)$, consider the 
intersection of the tangent spaces at $P$ of the nine quadrics corresponding to $[\pi]$.  Then this intersection is equal to the tangent plane $\TP$ of $[\pi]$ at $P$ as defined in Theorem~{\rm \ref{tgtpleqn}}.
\end{theorem}

\begin{proof}
By Theorems~\ref{transsplashes} and~\ref{propthm}, we can without loss of generality prove this for the \orsp\ $\Bpi$ coordinatised in Section~\ref{mapsection-Bpi2}, and the point $P=(0,0,1)$ of $\Bpi$. In $\PG(6,q)$, consider the nine  quadrics corresponding to $[\Bpi]$ which are given in equations (\ref{alt-second}), (\ref{alt-third}) and (\ref{alt-fourth}). 
We want to find  the set of lines through $P$ that meet each of these nine quadrics twice at $P$.
 Every line $\ell$ of $\PG(6,q)$ through $P$ has form $\ell=RP$ for some point 
$R=([u],[v],0)\in\si$, $u,v\in\gfqc$.  So the points of $\ell$ are of the form $P_s=P+sR=([su],[sv],1)$ where $s\in\gfq$.  
Substituting 
the point $P_s$ into the quadrics of (\ref{alt-second})
gives
\begin{eqnarray*}
(-\tau^{2q} su+\tau^{2q} sv+(\tau-\tau^q))^q(-su+sv)&=&(-\tau^{2q} su+\tau^2 sv+(\tau-\tau^q))(-su+sv)^q.
\end{eqnarray*}
This expression is a polynomial of degree two in $s$.  
The line $\ell=PR$ is tangent to the three quadrics of (\ref{alt-second}) if this expression has 
a repeated root $s=0$, that is, if the coefficient of $s$  is equal to zero.  That is,
$(\tau-\tau^q)^q(-u+v)=(\tau-\tau^q)(-u+v)^q$, and so $k=(-u+v)/(\tau-\tau^q)$ is in $\gfq$. Rearranging gives 
$
v=k(\tau-\tau^q)+u
$.
Substituting 
the point $P_s$ into the quadrics of (\ref{alt-third})
gives no constraints. 
Substituting 
the point $P_s$ into the quadrics of (\ref{alt-fourth})
and simplifying gives the constraint 
 $m=(-\tau^q u+ \tau v)/(\tau-\tau^q)$ lies in $\gfq$, and so  $v=(m(\tau-\tau^q)+\tau^qu)/\tau$. Equating this with the expression for $v$ obtained from (\ref{alt-second}) above gives  $u=m-k\tau$, and so $v=m-k\tau^q$. 
 Hence the line $\ell=PR$ is tangent to all nine quadrics when $R$ has form 
$$R=([u],[v],0)=([m-k\tau],[m-k\tau^q],0)
=m([1],[1],0)-k([\tau],[\tau^q],0).
$$
Thus the tangent space to $[\Bpi]$ at $P$  is the plane through $P$ and the line $\ell=
\langle([1],[1],0), ([\tau],[\tau^q],0)\rangle
$ of $\Sigma_\infty$. This is the same as the tangent plane $\TP$ to $[\Bpi]$ at $P$ calculated in  the proof of Theorem~\ref{tgtpleqn}.
\end{proof}

\section{Coordinatising the exterior splash and its covers}\Label{sec:coord-cover}

Let $\ES$ be an exterior splash of $\PG(1,q^3)$. In the Bruck-Bose representation, $\ES$ corresponds to a set of $q^2+q+1$ planes of the regular $2$-spread $\S$ in $\si\cong\PG(5,q)$. To simplify the notation, we use  the same symbol $\ES$ to denote both the points of the exterior splash on $\li$, and the planes of the exterior splash contained in $\S$.
In \cite{BJ-ext1}, we show that an exterior splash is projectively equivalent to a cover of the circle geometry CG$(3,q)$. Hence by 
Bruck~\cite{bruc73b}, there are two \emph{switching sets}
$\EX$, $\EY$ for $\ES$. That is, $\EX$ and $\EY$ consist of $q^2+q+1$ planes each, such that the planes of the three sets $\ES$, $\EX$ and $\EY$ each cover the same set of points. Further, planes from different sets meet in unique points, and planes in the same set are disjoint. 
The three sets $\ES,\EX,\EY$ are called 
 {\em hyper-reguli} in \cite{culbert,ostrom}.
In this article, we call the families $\EX$ and  $\EY$ \emph{covers of the exterior splash} $\ES$. 

In this section we take the \orsp\ $\Bpi$ 
coordinatised in Section~\ref{mapsection-Bpi2}, with exterior splash $\ES$, and use \cite{ostrom} to calculate the coordinates of the two covers of $\ES$. We denote the covers corresponding to $\Bpi$ by $\ET$ and $\EC$. The cover denoted $\ET$ in the following lemma is  called the {\em tangent cover} of $\ES$ with respect to $\Bpi$, or the tangent cover of $[\Bpi]$.
The cover  denoted $\EC$ in the following lemma is  called  {\em conic cover} of $\ES$ with respect to $\Bpi$, or the conic cover of $[\Bpi]$. The nomenclature for tangent covers comes from  Theorem~\ref{def:tgt-cover} which shows that the tangent planes $\TP$ of $[\Bpi]$ meet $\si$ in  lines that lie in distinct planes of the cover $\ET$.  The nomenclature for the conic cover comes from \cite{BJ-ext3} which shows that the planes in the cover
$\EC$ are related to the $(\Bpi,\li)$-special conics in $\Bpi$.

 \begin{lemma}\Label{coord-covers-new}
Let $\ES$ be the exterior splash of the exterior \orsp\ $\Bpi$ coordinatised in Section~{\rm\ref{mapsection-Bpi2}}. 
Let ${\mathscr K}=\{\k=\tau^{i(q-1)}\st 0\leq i<q^2+q+1\}$. In $\PG(6,q)$,  $\ES$ and its two covers $\ET,\EC$ have planes given by 
\begin{eqnarray*}
\ES&=&\{[S_\k]=\{([\k x],[x],0)\st x\in\gfqcs\}\st \k\in\mathscr{K}\}\\
\ET&=&\{[T_\k]=\{([\k x],[x^q],0)\st x\in\gfqcs\}\st \k\in\mathscr{K}\}\\
\EC&=&\{[C_\k]=\{([\k x],[x^{q^2}],0)\st x\in\gfqcs\}\st \k\in\mathscr{K}\}.
\end{eqnarray*}
\end{lemma}

\begin{proof} 
The points of $\li$ in $\PG(2,q^3)$ have coordinates $S_\k=(\k,1,0)$ for $\k\in\gfqc\cup\{\infty\}$. Hence in the Bruck-Bose representation of $\li$ in $\si\cong\PG(5,q)$,  planes of the regular $2$-spread $\S$ are given by
$
[S_\k]=\{([\k x],[x])\st x\in\gfqcs\},
$
for $\k\in\gfqc\cup\{\infty\}$.
Consider the homography $\beta$ (of order 3) of $\si\cong\PG(5,q)$ defined by:
\begin{eqnarray}\label{def:beta}
\beta: ([x],[y])\rightarrow([x],[y^q]).
\end{eqnarray}
We consider the action of $\beta$ on the planes of $[S_\k]$. For each $\k\in\gfqc\cup\{\infty\}$, define the planes $[T_k],[C_k]$ by  $\beta([S_\k])=[T_\k]$ and $\beta([T_\k])=[C_\k]$. That is,
$[T_\k]=\{([\k x],[x^q])\st x\in\gfqcs\}$, and $
[C_\k]=\{([\k x],[x^{q^2}])\st x\in\gfqcs\}$.

We now consider the exterior \orsp\ $\Bpi$ coordinatised in Section~\ref{mapsection-Bpi2} which by Theorem~\ref{phi-prop} has exterior splash $\ES=\{S_k=(\k,1,0)\st \k\in\mathscr{K}\}\subset\li$, and carriers $S_\infty=(1,0,0)$, $S_0=(0,1,0)$. 
Note that
in $\PG(5,q)$, the carriers of $\Bpi$ lie in each of the three sets of planes, as $[S_0]=[T_0]=[C_0]$ and 
 $[S_\infty]=[T_\infty]=[C_\infty]$.
 In $\PG(5,q)$, we have $\ES=\{[S_\k]\st \k\in{\mathscr K}\}$. Let $\ET=\{[T_\k]\st \k\in{\mathscr K}\}$ and
$\EC=\{[C_\k]\st \k\in{\mathscr K}\}$, then 
$\beta:\ES\mapsto\ET\mapsto\EC$.
By 
\cite{ostrom}, the sets $\ES$, $\ET$, $\EC$ cover the same set of points. Moreover, planes in the same set are disjoint, and planes from different sets meet in one point. That is, $\ET$ and $\EC$ are the two covers of $\ES$.
\end{proof}

The next lemma demonstrates a useful homography of $\PG(6,q)$ that acts regularly on the cover planes in each cover of $\Bpi$. 
Recall that $\tau$ is a zero of the primitive polynomial $x^3-t_2x^2-t_1x-t_0$.

\begin{lemma}\Label{theta-action}
Let $\ES$ be the  exterior splash of the exterior \orsp\ $\Bpi$ coordinatised in Section~{\rm \ref{mapsection-Bpi2}} with covers $\EC$ and $\ET$ coordinatised in Lemma~{\rm \ref{coord-covers-new}}.
Consider  the homography $\Theta\in\PGL(7,\r)$ with $7\times 7$ matrix
\[
\begin{pmatrix}M&0&0\\0&M&0\\0&0&1\end{pmatrix}, \quad{\rm where}\ \ 
M=\begin{pmatrix}0&0&t_0\\1&0&t_1\\0&1&t_2\end{pmatrix}.
\]
Then in $\PG(6,q)$, $\Theta$ fixes each plane of the regular $2$-spread $\S$; maps the cover plane $[C_k]\in\EC$ to $[C_{\tau^{1-q} k}]\in\EC$; and maps the cover plane $[T_k]\in\ET$ to $[T_{\tau^{1-q^2} k}]\in\ET$, $\k\in\mathscr{K}$.
\end{lemma}

\begin{proof} It is straightforward to show that 
 $\Theta$ fixes the planes of the regular $2$-spread $\S$ (so it also fixes the planes of the exterior splash $\ES$). 
 In fact
$\langle\Theta\rangle$
acts regularly on the set of points, and on the set of lines, of each spread element.
Note that $M$ is the matrix $M_\tau$ defined in Section~\ref{sec:eg:Mk}, and so $M[x]=[\tau x]$. 
Consider the action of $\Theta$ on a point of the cover plane $[C_k]\in\EC$ coordinatised in  Lemma~\ref{coord-covers-new}.
We have 
$([kx],[x^{q^2}],0)^\Theta=([\tau kx],[\tau x^{q^2}],0)\equiv([\tau^{1-q}\k(\tau^q x)],[(\tau^q x)^{q^2}],0)$ which lies in the cover plane $[C_{\tau^{1-q} \k}]$ of $\EC$. 
 Similarly a point $([\k x],[x^q],0)$ in the cover plane $[T_k]\in\ET$ maps under $\Theta$ to the point $([\tau^{1-q^2}\k(\tau^{q^2} x)],[(\tau^{q^2} x)^q],0)$ which lies in the cover plane $[T_{\tau^{1-q^2}\k}]$ of $\ET$. 
\end{proof}

\begin{theorem}\Label{def:tgt-cover}
Let $P$ be a point of an exterior \orsp\ $\pi$. In $\PG(6,q)$, the tangent plane $\TP$ at $P$ to $[\pi]$ 
meets $\Sigma_\infty$ in a line that lies in a  plane  of the tangent cover $\ET$ of $[\pi]$. Moreover, distinct points of $\pi$ correspond to distinct cover planes of $\ET$.
\end{theorem}

\begin{proof} By Theorems~\ref{transsplashes} and~\ref{propthm}, we can without loss of generality prove this result for the \orsp\ $\Bpi$ coordinatised in Section~\ref{mapsection-Bpi2} and the point $P=(0,0,1)\in\Bpi$. 
In $\PG(6,q)$, let $\TP$ be the tangent plane at $P$, the line 
$\ell=\TP\cap\si$ was calculated in the proof of Theorem~\ref{tgtpleqn} to be 
$\ell=\{a([1],[1],0)+b([\tau],[\tau^q],0)\st a,b\in\gfq\}$. The points of $\ell$ all lie in the plane $[T_1]=\{[x],[x^q],0)|x\in\gfqcs\}$, which by Theorem~\ref{coord-covers-new} is a plane of the tangent cover $\ET$ of $\Bpi$.
The collineation of Lemma~\ref{theta-action} is transitive on the cover planes of $\ET$, hence each cover plane contains a line of a distinct tangent plane. Hence there 
is a one-to-one correspondence between points of $\pi$ and cover planes of $\ET$. 
\end{proof}

\section{Transversal lines of covers}\Label{sec:coords-covers-new}

Recall that a regular $2$-spread in $\PG(5,q)$ has three (conjugate skew) transversals in $\PG(5,q^3)$ which meet each (extended) plane of $\S$. In this section we consider an exterior splash $\ES\subset\S$, and show in Lemma~\ref{unique-trans} that  the transversals of the $2$-spread $\S$ are the only lines of $\PG(5,q^3)$ that meet every extended plane of  $\ES$.  We then consider the two sets of cover planes $\ET$ and $\EC$. We show in Theorem~\ref{STC-coord} that  each can be extended to regular $2$-spread, and we calculate the coordinates of the corresponding transversal lines  in Theorem~\ref{gs-gt-gc}. We then show in Theorem~\ref{splash-carriers} that the nine transversals of $\ES$, $\EC$ and $\ET$ can be used to characterise the carriers of the exterior splash $\ES$. Finally, in Theorem~\ref{trans-for-all-covers}, we discuss the transversal lines in the situation when $\li$ is partitioned into exterior splashes with common carriers.

\subsection{The exterior splash and its covers have unique transversals}

If $\X$ is a set in $\PG(6,q)$ (such as a line, a plane, or a conic), then we denote its natural extension to $\PG(6,q^3)$ by $\X^*$.
Let $\S$ be the regular $2$-spread in $\si$ of the Bruck-Bose representation in $\PG(6,q)$. 
If we extend the planes of $\S$ to $\PG(6,q^3)$, yielding $\S^*$,  then there are exactly three transversal lines to $\S^*$, that is, three lines that meet every plane of $\S^*$. These three lines are conjugate and skew.
We now consider an exterior splash $\ES\subset\S$ and extend the planes of $\ES$ to $\PG(6,q^3)$, yielding $\ES^*$. We show that there are exactly three lines of $\PG(6,q^3)$ that meet every plane of $\ES^*$, namely the three transversals of $\S$.

\begin{lemma}\Label{unique-trans}
 Let $\S$ be a regular $2$-spread in $\PG(5,q)$, and let $\ES\subset\S$ be an exterior splash. 
 In the cubic extension $\PG(5,q^3)$, there are 
 exactly three transversals to $\ES$, namely the three transversals of  $\S$.  Hence  $\ES$ lies in a unique regular $2$-spread, namely $\S$.
\end{lemma}

\begin{proof}
The three conjugate transversal lines of the regular $2$-spread $\S$, denoted $g_\ES,\gs^q,\gs^{q^2}$, are also transversals of $\ES$. 
Suppose there is a fourth transversal line $\ell$ of $\ES$. Then the four lines $\gs,\gs^q,\gs^{q^2},\ell$ are pairwise skew. Further, these four lines are ruling lines of  a unique $2$-regulus ${\mathscr R}$ of $\si^*\cong\PG(5,q^3)$, which contains the set of extended planes $\ES^*$. Now consider two planes $[L],[M]\in\ES$, the corresponding points $L,M$ of $\li$ in $\PG(2,q^3)$ lie in two \orsls\ contained in $\ES$ by \cite[Corollary 15]{lavr10}. Hence by Theorem~\ref{sublinesinBB}, $[L],[M]$ lie in two $2$-reguli $\R_1,\R_2$ which are contained in $\ES$. Let $P$ be a point in $[L]$, then there are unique  lines $m_1,m_2$ through $P$ that are ruling lines of $\R_1,\R_2$ respectively. Now $\R_1$, $\R_2$ lie in $\ES$, and so lie in ${\mathscr R}$, so the extended lines $m_i^*$, $i=1,2$,  are two ruling lines of ${\mathscr R}$ that meet in a point $P$, a contradiction. Hence the line $\ell$ cannot exist. That is, there are only three transversal lines to $\ES$, and these are necessarily the transversals of $\S$. 
\end{proof}

Analogous to Lemma~\ref{unique-trans} about an exterior splash $\ES$, we will show that the two covers of $\ES$ each have exactly three conjugate transversal lines in $\PG(6,q^3)$. 

\begin{theorem}\Label{STC-coord}
In $\PG(5,q)$, let $\ES$ be an exterior splash with covers $\ET$ and $\EC$. Then in the cubic extension $\PG(5,q^3)$,
\begin{enumerate}
\item  The cover $\ET$ has exactly three transversal lines:  $\gt,\gt^q,\gt^{q^2}\subseteq\PG(5,q^3)\setminus\PG(5,q)$, and so $\ET$ lies in a unique regular $2$-spread.
\item  The cover $\EC$ has exactly three transversal lines: $\gc,\gc^q,\gc^{q^2}\subseteq\PG(5,q^3)\setminus\PG(5,q)$, and so $\EC$ lies in a unique regular $2$-spread.
\end{enumerate}
\end{theorem}

\begin{proof}
By Theorem~\ref{transsplashes}, we can without loss of generality prove this for the exterior splash of the \orsp\ $\Bpi$ coordinatised in Section~\ref{mapsection-Bpi2}. We use the coordinatisation in $\PG(5,q)$ of the exterior splash $\ES$ of $\Bpi$ and the two covers $\ET$, $\EC$ given in Lemma~\ref{coord-covers-new}. 
The homography $\beta$ of $\PG(5,q)$ defined in (\ref{def:beta}) can be regarded as a homography of $\PG(5,q^3)$. By the proof of Lemma~\ref{coord-covers-new},  $\beta$ maps the spread plane $[S_\k]$, $k\in\gfqc\cup\{\infty\}$ to a plane $[T_\k]$ of $\si$ (note that $[T_k]\in\ET$ if $k\in{\mathscr K}$). Thus $\beta$ maps the transversal line $g_\ES$ to a line $g_\ET$ that meets every extended plane $[T_\k]^*$, $\k\in\gfqc\cup\{\infty\}$. Hence $g_\ET$ 
meets each (extended) cover plane of $\ET=\{[T_k]\st k\in{\mathscr K}\}$, and so  $g_\ET$ 
 is a transversal line of $\ET$. Thus the conjugate lines $g_\ET^q,g_\ET^{q^2}$ are also transversals of $\ET$. The argument that these are the only transversals of $\ET$ 
is similar to the proof of Lemma~\ref{unique-trans}.
Finally, note that the set 
$\{[T_\k]\st \k\in\gfqc\cup\{\infty\}\}$ is a regular $2$-spread with transversals $g_\ET,g_\ET^q,g_\ET^{q^2}$, and so it is the unique regular $2$-spread containing $\ET$. As $\beta$ maps the  plane $[T_k]$ to the plane $[C_k]$, and maps the line $\gt$ to the line $\gc$, a similar result holds for the conic cover $\EC$.
\end{proof}

Later we will need the coordinates of  the point of intersection of the transversal lines with the corresponding cover planes, and we calculate these next.

\begin{theorem}\Label{gs-gt-gc} Let $\Bpi$ be the \orsp\ coordinatised in Section {\rm\ref{mapsection-Bpi2}} with exterior splash $\ES$ and covers $\EC,\ET$. Let $p_0=
t_1+t_2\tau-\tau^2=-\tau^q\tau^{q^2}$, $p_1=t_2-\tau=\tau^q+\tau^{q^2}$, $p_2=-1$, and $\eta= p_0+p_1\tau+p_2\tau^2$. Let $A_1=(p_0,p_1,p_2,0,0,0,0)$, $A_2=(0,0,0,p_0,p_1,p_2,0)$. Then  in $\PG(6,q^3)$, 
\begin{enumerate}
\item One transversal line of $\ES$ is $g_\ES=\langle A_1,A_2\rangle$, and $g_\ES\cap[S_\k]^*=\k A_1+A_2$.
\item One transversal line of $\ET$ is $g_\ET=\langle A_1,A_2^{q^2}\rangle$, and $g_\ET\cap[T_\k]^*=\k A_1+\eta^{1-q^2}A_2^{q^2}$.
\item One transversal line of $\EC$ is $g_\EC=\langle A_1,A_2^{q}\rangle$, and $g_\EC\cap[C_\k]^*=\k A_1+\eta^{1-q}A_2^{q}$.
\end{enumerate}
\end{theorem}

\begin{proof} 
We use the coordinatisation in $\PG(5,q)$ of the exterior splash $\ES$ of $\Bpi$ and the two covers $\ET$, $\EC$ given in Lemma~\ref{coord-covers-new}.
Lemma~\ref{S-coord-g} shows that $\gs=\langle A_1,A_2\rangle$ is a  transversal line for the regular 2-spread $\S$,  where
$A_1=(p_0,p_1,p_2,0,0,0)=(A,[0])$ and $A_2=(0,0,0,p_0,p_1,p_2)=([0],A)$. Hence $\gs=\langle A_1,A_2\rangle$ is a transversal line for the exterior splash $\ES$. 
 The planes of the regular $2$-spread $\S$ are $[S_\k]=\{([\k x],[x])\st x\in\gfqcs\}$, $k\in\gfqc\cup\{\infty\}$. We first show that the extended plane $[S_\k]^*$  meets the line $g_\ES$ in the point $\k A_1+A_2$.
Consider the point $P=p_0([\k ],[1])+p_1([\k \tau],[\tau])+p_2([\k \tau^2],[\tau^2])$ of $\PG(5,q^3)$, and note that $P\in[S_\k]^*$. Using the matrix $\MM_\k$ defined in Section~\ref{sec:eg:Mk}, we have
$
P=p_0(\MM_\k [1],[1])+p_1(\MM_\k [\tau],[\tau])+p_2(\MM_\k [\tau^2],[\tau^2])
=(\MM_\k A,A)
=(\k A,A)$ by (\ref{JA-eqn}). Hence $P=
\k A_1+A_2$ which lies in $g_\ES=\langle A_1,A_2\rangle$, that is,  $P$ is the intersection of $g_\ES$ and $[S_\k]^*$ proving part 1.

Consider the homography $\beta$ defined in (\ref{def:beta}), acting on $\PG(5,q^3)$.
The proof of Theorem~\ref{STC-coord} shows 
that $\beta$ maps $g_\ES$ to $g_\ET$, and maps $\gt$ to $\gc$.
Each element $y\in\gfqc'$ can be considered as a point  $[y]$ in $\PG(2,q)$. The collineation of $\PG(2,q)$ mapping the point  
$[y]$ to $[y^q]$ is a homography, and can be represented using a matrix $N$ with entries in $\gfq$.
We omit the transpose notation, and write 
 $N[y]=[y^q]$. Hence we can write the collineation $\beta$ as $\beta([x],[y])=([x],N[y])$. Clearly $\beta(A_1)=A_1$, we show that $\beta(A_2)=A_2^{q^2}$.  
Recall the point $A=(p_0,p_1,p_2)=p_0[1]+p_1[\tau]+p_2[\tau^2]$,
so $NA= p_0[1]+p_1[\tau^q]+p_2[\tau^{2q}]$. 
Using the matrix  $\MM_k$ from Section~\ref{sec:eg:Mk}, it is straightforward to write this as $NA=
\big(p_0^{q^2}I+p_1^{q^2}\MM_\tau+p_2^{q^2}M_{\tau^2}\big)^q\,[1]$.
Now $(p_0^{q^2}I+p_1^{q^2}\MM_\tau+p_2^{q^2}M_{\tau^2})[1]=A^{q^2}$, and $(p_0^{q^2}I+p_1^{q^2}\MM_\tau+p_2^{q^2}M_{\tau^2})A^{q^2}=\eta^{q^2}A^{q^2}$ by (\ref{JA-eqn}). So 
repeated use of  (\ref{JA-eqn}) yields $NA=\eta^{q^2(q-1)}A^{q^2}=\eta^{1-q^2}A^{q^2}$.
Further, as $N$ is over $\gfq$, we have
\begin{eqnarray}\label{eqn-NA}
NA=\eta^{1-q^2}A^{q^2},\quad NA^q=\eta^{q-1} A,\quad NA^{q^2}=\eta^{q^2-q} A^q.\end{eqnarray}
Hence $\beta( \k A_1+A_2)= \k A_1+\eta^{1-q^2}A_2^{q^2}$. As $\beta:\gs\mapsto\gt$, we have  $g_\ET\cap [T_\k]^*= \k A_1+\eta^{1-q^2}A_2^{q^2}$ and    $g_\ET=\langle A_1,A_2^{q^2}\rangle$, proving part 2.  Similarly, calculating $\beta( \k A_1+\eta^{1-q^2}A_2^{q^2})=\k A_1+\eta^{1-q^2+q^2-q}A_2^q=\k A_1+\eta^{1-q}A_2^q$, and using $\beta:\gt\mapsto\gc$, we get $g_\EC\cap [C_\k]^*=\k A_1+\eta^{1-q}A_2^q$ and $g_\EC=\langle A_1,A_2^{q}\rangle$. 
\end{proof}

We can use the transversals of the covers $\ET$ and $\EC$ to generalise the notion of $\S$-special conics and twisted cubics in $\PG(6,q)$ defined in Definition~\ref{def:S-special}. We define $\EC$-special here, $\ET$-special is similarly defined.

\begin{definition}\Label{C-special}
\begin{enumerate}
\item 
A {\em
   $\EC$-special conic}  is a non-degenerate conic $\C$ contained in a plane of $\EC$, such
that 
the extension of $\C$ to $\PG(6,q^3)$
 meets the transversals of
 $\EC$. 
 \item A {\em $\EC$-special twisted cubic} is a
twisted cubic  $\N$ in a $3$-space
 of $\PG(6,q)\backslash\si$ about a plane of $\EC$, such that the extension of $\N$ to $\PG(6,q^3)$ meets the transversals of $\EC$. 
\end{enumerate}
\end{definition}

\subsection{Characterising the carriers of an exterior splash in $\PG(6,q)$}

Let $\S$ be a regular $2$-spread of $\PG(5,q)$, and let $\ES$ be an exterior splash contained in $\S$, with covers $\EC$ and $\ET$. Then we can characterise the carriers of $\ES$ in terms of the nine transversals 
of $\ES$, $\EC$ and $\ET$.

\begin{theorem}\Label{splash-carriers}
Let $\S$ be a regular $2$-spread of $\PG(5,q)$, and let $\ES\subset\S$ be an exterior splash with covers $\EC$, $\ET$, whose corresponding triples of transversal lines are $\gs,\gs^q,\gs^{q^2}$,  $\gc,\gc^q,\gc^{q^2}$,  $\gt,\gt^q,\gt^{q^2}$ respectively.  Then the carriers of $\ES$ are the only two planes of $\S$ whose extension to $\PG(5,q^3)$ meets all nine transversal lines.
\end{theorem}
\begin{proof}  
By Theorem~\ref{transsplashes}, we can without loss of generality show this for the exterior splash $\ES$ of the exterior \orsp\ $\Bpi$ coordinatised in Section~\ref{mapsection-Bpi2}, with carriers $E_1=(1,0,0)$, $E_2=(0,1,0)$. In $\PG(6,q)$, the transversal lines $\gs,\gs^q,\gs^{q^2}$ each meet the carriers $[E_1]$, $[E_2]$ of $\ES$.
We use the notation for planes $[S_k]\in\S$, $[T_k]\in\ET$ and $[C_k]\in\EC$ from Lemma~\ref{coord-covers-new}.
By the proof of Theorem~\ref{STC-coord}, in the cubic extension $\PG(5,q^3)$, the transversal lines $\gt,\gt^q,\gt^{q^2}$ meet each plane $[T_k]$, $k\in\gfqc\cup\{\infty\}$; and the transversal lines $\gc,\gc^q,\gc^{q^2}$ meet each plane $[C_k] $, $k\in\gfqc\cup\{\infty\}$.
The carriers of $\ES$ satisfy $[E_2]=[S_0]=[T_0]=[C_0]$ and 
 $[E_1]=[S_\infty]=[T_\infty]=[C_\infty]$. Hence in the cubic extension $\PG(5,q^3)$, all nine transversal lines meet the carriers of $\ES$.

We now show that no other plane of the regular $2$-spread $\S$ meets all nine transversal lines. 
We use  the 
homography with matrix $\MM_k$ defined in Section~\ref{sec:eg:Mk}.
A plane of the regular $2$-spread $\S$ distinct from $[E_1]$, $[E_2]$ has form 
$[S_k]=\{([kx],[x],0)\st x\in\gfqcs\}$, for some $k\in\gfqcs$. This plane is spanned by the three points 
$S_{0,k} = ([k],[1],0)=(M_k(1,0,0),(1,0,0))$, 
$S_{1,k} = ([k\tau],[\tau],0)=(M_k(0,1,0),(0,1,0))$  and
$S_{2,k} = ([k\tau^2],[\tau^2],0)=(M_k(0,0,1),(0,0,1))$.
Hence the extension $[S_k]^*$ to $\PG(5,q^3)$ contains the points $S_{k,j}=c_0S_{0,j}+c_1S_{1,j}+c_2S_{2,j}$ where $c_i\in\gfqc$, not all zero. 
By Theorem~\ref{gs-gt-gc}, a general point $X$ on the transversal line $g_\ET$ has coordinates $X=rA_1+ A_2^{q^2}=(r p_0,r p_1,r p_2,p_0^{q^2},p_1^{q^2},p_2^{q^2})$, for some $r\in\gfqc\cup\{\infty\}$. 
Now $S_{j,k}=X$ if and only if $c_i=p_i^{q^2}$, $i=0,1,2$, and $M_k(c_0,c_1,c_2)=r(p_0,p_1,p_2)$. That is, $M_kA^{q^2}=rA$. However, by (\ref{JA-eqn}), $M_kA^{q^2}=k^{q^2}A^{q^2}$, so there are no solutions to $c_0,c_1,c_2$.
Hence the transversal line $g_\ET$ does not meet any further plane of the regular $2$-spread $\S$, and so $\gt^q,\gt^{q^2}$ do not meet any further plane of $\S$. A similar argument shows that the lines $\gc,\gc^q,\gc^{q^2}$ 
do not meet any further plane of the regular $2$-spread $\S$.
\end{proof}

\subsection{Transversal lines of exterior splashes with common carriers}
\Label{sec:cov-spl-car}

As exterior splashes are equivalent to covers of the circle geometry $CG(3,q)$,  there are $q-1$ disjoint exterior splashes on $\li$ with common carriers $\car_1,\car_2$. We show that in $\PG(6,q)$, the covers of these disjoint exterior splashes have common transversals.

\begin{theorem}\Label{trans-for-all-covers}
Let\, $\ES_0,\ldots,\ES_{q-1}$ be $q-1$ disjoint exterior splashes on $\li$ with common carriers $E_1,E_2$, and let exterior splash $\ES_j$ have covers $\EC_j$, $\ET_j$. 
Then the covers $\EC_0,\ldots,\EC_{ q-1}$ have  common transversal lines $\gc,\gc^q,\gc^{q^2}$, and 
 the covers $\ET_0,\ldots,\ET_{ q-1}$ have  common transversal lines $\gt,\gt^q,\gt^{q^2}$.
\end{theorem}

\begin{proof}
By Theorem~\ref{transsplashes}, we can without loss of generality prove this for the \orsp\ $\Bpi$ coordinatised in Section~\ref{mapsection-Bpi2}. Let $\KK=\{ k \in \gfqcs\st k^{q^2+q+1}=1\}
=\{\k=\tau^{i(q-1)}\st 0\leq i<q^2+q+1\}$. Recall that $\Bpi$ has carriers $E_1=(1,0,0)$, $E_2=(0,1,0)$, and exterior splash $\ES_0=\{S_{k,0}=(k,1,0)\st k\in\KK\}$.
Let  $\KK_j=\tau^j\KK$,  for $j=0\ldots, q-2$, be the $q-1$ cosets of $\KK$ in $\gfqcs$.  
Let $\ES_j=\{S_{k,j}=(k,1,0)\st k\in\KK_j\}$, $0\le j\le q-2$. 
Consider the homography 
$\xi$ acting on $\li$ that maps the point $(x,y,0)$ to $(\tau x,y,0)$. Then $\xi$ fixes $E_1,E_2$, maps
$\ES_j$ to $\ES_{j+1}$ ($0\le j\le q-3)$, and maps $\ES_{q-2}$ to $\ES_0$. Hence $\ES_0,\ldots,\ES_{q-1}$ are the $q-1$ disjoint exterior splashes on $\li$ with carriers $(1,0,0)$ and $(0,1,0)$.

In $\si\cong\PG(5,q)$, we have  planes  
$[S_{k,j}]=\{([kx],[x])\st x\in\gfqcs\}\in\ES$, and define planes  
$[T_{k,j}]=\{([k x],[x^q])\st x\in\gfqcs\}$, and 
$[C_{k,j}]=\{([k x],[x^{q^2}])\st x\in\gfqcs\}$, for $k\in\KK_j$ .
So $\ES_j=\{[S_{k,j}],k\in\KK_j\}$, and define
$\ET_j=\{[T_{k,j}],k\in\KK_j\}$ and $\EC_j=\{[C_{k,j}],k\in\KK_j\}$. Note that $\ET_0$, $\EC_0$ are the covers of the exterior splash $\ES_0$ of $\Bpi$. 
Now consider the map $\theta_\tau$ of $\PG(5,q)$ acting on $\si$ defined in Section~\ref{sec:eg:Mk}, it maps $\ES_j$ to $\ES_{j+1}$; $\ET_j$ to $\ET_{j+1}$; and  $\EC_j$ to $\EC_{j+1}$. 
Hence $\ET_j$ and $\EC_j$ are covers for $\ES_j$.
By 
Theorem~\ref{gs-gt-gc}, the transversal line of $\ET_0$ is $\gt=\langle A_1,A_2^{q^2}\rangle$. 
Using (\ref{JA-eqn}), we see that 
the homography $\theta_\tau$ fixes $\gt$, and so $\gt$ is a transversal for all $\ET_j$. So $\gt,\gt^q,\gt^{q^2}$ are transversal lines of $\ET_j$ for each $j=0,\ldots,q-2$. Similarly,  
$\gc,\gc^q,\gc^{q^2}$ are transversal lines of $\EC_j$ for each $j=0,\ldots,q-2$.
\end{proof}


\begin{remark} {\rm We can interpret this result using the terminology of \cite{culbert}. 
We can partition the planes of a regular $2$-spread into $q-1$ disjoint hyper-reguli with common carriers. Each hyper-reguli has two replacement hyper-reguli, which correspond to our conic and tangent covers. If we replace all $q-1$ hyper-reguli of $\S$ with hyper-reguli of the {\em same type} (that is, all belonging to $\EC$, or all belonging to $\ET$), then the resulting $2$-spread has transversals either $\gc,\gc^q,\gc^{q^2}$ or $\gt,\gt^q,\gt^{q^2}$, and so 
is regular. Hence the resulting Andr\'e plane is Desarguesian. If we replace all the hyper-reguli of $\S$ with a combination of hyper-reguli from each type, then the resulting $2$-spread is not regular, and so the resulting Andr\'e plane is non-Desarguesian.}
\end{remark}

\section{A characterisation of the sublines of an exterior splash}\Label{sec:orsl-BB}

Let $\pi$ be an exterior \orsp\ of $\PG(2,q^3)$ with exterior splash $\ES$ on $\li$. The 
order-$q$-sublines of $\ES$ are studied in \cite{lavr10}, and  are characterised with respect to geometric objects of $\pi$ in \cite{BJ-ext1}. 
In this section we characterise the \orsls\ of $\ES$ with respect to the covers of $\ES$ and their transversal lines.

There are $2(q^2+q+1)$ \orsls\ in an exterior splash which lie in two families of size $q^2+q+1$.
By \cite[Theorem 5.2]{BJ-ext1}, the two families are related to the associated  exterior \orsp\ $\pi$ as follows. 
 If $A$ is a point of $\pi$, then the pencil of $q+1$ lines of $\pi$ through $A$ meets $\li$ in an \orsl\ of $\ES$, called a {\em $\pi$-\psline} of $\ES$. 
 If $\Gamma$ is a $(\pi,\li)$-special-dual conic of $\pi$, then the lines of $\Gamma$ meet $\li$ in an \orsl\ of $\ES$, called a {\em $\pi$-\dcsline} of $\ES$. 
Note that we need to define these two classes with respect to an associated \orsp\ $\pi$, since in \cite[Theorem 5.3]{BJ-ext1}, we show that the two families of \orsls\ of $\ES$ can act as either class by considering different associated \orsps.

We now consider the interaction in $\PG(6,q)$ of these two families of \orsls\ with the two covers of $\ES$. We show in Theorem~\ref{psline-special} that each family meets planes from one cover in lines, and planes from  the other cover in conics. Further, Theorem~\ref{psline-converse} shows that  the converse is true, and so we have a characterisation of the \orsls\ of $\ES$. Moreover, Theorem~\ref{psline-conics-special} shows that the conics concerned in each case are special.

Suppose $\R$ is a $2$-regulus in $\PG(5,q)$, and consider a plane $\alpha$ that meets $\R$ in a set of $q+1$ points. Then an easy counting argument shows that these points form either a line or a conic in $\alpha$. We abbreviate this to `$\R$ meets $\alpha$  in a line or a conic'. 

\begin{theorem}\Label{psline-special}
Let $\pi$ be an exterior \orsp\ with exterior splash $\ES$, conic cover $\EC$, and  tangent cover $\ET$. 
\begin{enumerate}
\item A $\pi$-\psline\ of\, $\ES$ corresponds  in $\PG(6,q)$ to a $2$-regulus that  meets 
 each plane of   $\ET$ in a  distinct line, and meets each plane of $\EC$ in a conic.
\item A $\pi$-\dcsline\ of\, $\ES$ corresponds in $\PG(6,q)$ to a $2$-regulus that meets 
each plane of $\ET$ in a conic, and
meets each plane of  $\EC$ in a  distinct line. 
\end{enumerate}
\end{theorem}

\begin{proof} 
Let $P$ be a point in the exterior \orsp\ $\pi$, and let $d$ be the corresponding $\pi$-\psline\ of $\ES$. By Theorem~\ref{sublinesinBB}, in $\PG(6,q)$, $[d]$ is a $2$-regulus contained in $\ES$. 
Consider the tangent plane $\TP$ to $[\pi]$ at $P$. By Theorem~\ref{tgtpleqn}, the lines of $\TP$ through $P$  meet $\si$ in points that lie in distinct planes of the $2$-regulus $[d]$. Hence $\TP\cap\si$ is a ruling line of the $2$-regulus $[d]$. By Theorem~\ref{def:tgt-cover}, this ruling line $\TP\cap\si$ lies in a tangent cover plane. The homography $\Theta$ of Lemma~\ref{theta-action} fixes the planes of $[b]$ and is transitive on the cover planes of $\ET$. Hence each ruling line of $[b]$ meets a unique cover plane of $\ET$.

A straightforward geometric argument shows that planes of $\ET,\EC$ meet a $2$-regulus of $\ES$ in a line or a conic. Hence a conic cover plane meets the $2$-regulus $[d]$ in a conic. 
As there are $q^2+q+1$ $\pi$-\psline s of $\ES$, every line in a plane of $\ET$ is a ruling line for some $2$-regulus corresponding to a $\pi$-\psline. Hence 
if $[d']$ is a $2$-regulus of $\ES$ corresponding to a $\pi$-\dcsline, then planes of $\ET$ meet $[d']$ in conics, and so planes of $\EC$ meet $[d']$ in ruling lines of $[d']$. Moreover, applying the homography of Lemma~\ref{theta-action} shows that each ruling line of $[d']$ lies in a unique conic cover plane.
\end{proof}

By Theorem~\ref{sublinesinBB}, there is a one to one correspondence between the \orsls\ of $\ES$ in $\PG(2,q^3)$, and the $2$-reguli contained in $\ES$ in $\PG(6,q)$. Hence the converse of Theorem~\ref{psline-special} is also true, and so we have a characterisation of \orsls\ of $\ES$ relating to the cover planes of the associated \orsp.

\begin{theorem}\Label{psline-converse}
Let $\pi$ be an exterior \orsp\ with exterior splash $\ES$, conic cover $\EC$, and  tangent cover $\ET$. 
\begin{enumerate}
\item A $\pi$-\psline\ of\, $\ES$ corresponds   to a $2$-regulus that  meets 
 each plane of   $\ET$ in a  distinct line. Conversely, a $2$-regulus contained in $\ES$ that meets some plane of $\ET$ in a line corresponds to a $\pi$-\psline\ of\, $\ES$.
 \item A $\pi$-\psline\ of\, $\ES$ corresponds   to a $2$-regulus that  meets 
 each plane of    $\EC$ in a conic. Conversely, a $2$-regulus contained in $\ES$ that meets some plane of  $\EC$ in a conic corresponds to a $\pi$-\psline\ of\, $\ES$.
 \item A $\pi$-\dcsline\ of\, $\ES$ corresponds  to a $2$-regulus that meets 
each plane of $\ET$ in a conic.
Conversely,  a $2$-regulus contained in $\ES$ that meets some plane of  $\ET$ in a conic corresponds to a $\pi$-\dcsline\ of\, $\ES$.
 \item A $\pi$-\dcsline\ of\, $\ES$ corresponds  to a $2$-regulus that meets 
each plane of $\EC$ in a  distinct line. Conversely, a $2$-regulus contained in $\ES$ that meets some plane of  $\EC$ in a   line corresponds to a $\pi$-\dcsline\ of\, $\ES$.
\end{enumerate}
\end{theorem}

In fact, we can give a stronger characterisation of the \orsls\ of $\ES$, namely that the conics of Theorem~\ref{psline-special} are special with respect to the associated cover. In order to prove that the conics are special,  we need to introduce coordinates, and the proof  is calculation intensive.

\begin{theorem}\Label{psline-conics-special}
Let $\pi$ be an exterior \orsp\ with exterior splash $\ES$, conic cover $\EC$, and  tangent cover $\ET$. 
\begin{enumerate}
\item A $2$-regulus of $\ES$ corresponding to a $\pi$-\psline\ of $\ES$ meets each plane of $\EC$ in a $\EC$-special conic.
\item A $2$-regulus of $\ES$ corresponding to a $\pi$-\dcsline\ of $\ES$ meets each plane of $\ET$ in a $\ET$-special conic.
\end{enumerate}
\end{theorem}

\begin{proof}
By Theorem~\ref{transsplashes}, we can without loss of generality prove this for the exterior \orsp\ $\Bpi$ coordinatised in Section~\ref{mapsection-Bpi2}.
We start with the \orsp\ $\pi_0=\PG(2,q)$ and the  line 
$\ell=[-\tau\tau^q,\tau+\tau^q,-1]$ which is exterior to $\pi_0$. Note that using the notation for $p_0,p_1,p_2$ given in Theorem~\ref{gs-gt-gc}, we have $\ell=[p_0^{q^2},
p_1^{q^2},
p_2^{q^2}]$.
A line of $\pi_0$ has coordinates $[l,m,n]$ for $l,m,n\in\gfq$, and  meets $\ell$ in the point $W'_{l,m,n}=
(-n(\tau+\tau^q)-m,\ l-n\tau\tau^q,\ m\tau\tau^q+l(\tau+\tau^q))$.
We  apply the homography $\sigma$ of Section~\ref{mapsection-Bpi2} with matrix $K$ to map $\pi_0$, $\ell$ to $\Bpi$, $\li$, respectively.
The point $W'_{l,m,n}$ of $\ell$ maps to the point $W_{l,m,n}=KW'_{l,m,n}= (l+m\tau  +n\tau^2,\ l+m\tau^q+ n\tau^{2q},\ 0)$ of $\li$. Writing $\varepsilon=\varepsilon_{l,m,n}=l+m\tau+ n\tau^{2}$, we have $W_\varepsilon=W_{l,m,n}=(\varepsilon,\varepsilon^{q},0)\equiv(\varepsilon^{1-q},1,0)$. Using the notation from Lemma~\ref{coord-covers-new}, this is the point $S_{\varepsilon^{1-q}}\in\li$.
In $\PG(6,q)$, $W_\varepsilon$ corresponds to the spread plane  
$[W_\varepsilon]=[W_{l,m,n}]=\{([\varepsilon x],[\varepsilon^{q} x],0)\equiv([\varepsilon^{1-q} x],[ x],0)\st x\in\gfqcs\}=[S_{\varepsilon^{1-q}}].$

Fix a point $P=(a,b,c)$ of $\pi_0$, so $a,b,c,\in\gfq$, not all zero. Let 
$\mathcal L=\{
(l,m,n)\st l,m,n\in\gfq {\rm \ not \ all\ zero,\ and\ } la+mb+nc=0\}.$
The $q+1$ lines of $\pi_0$ through $P$ have coordinates 
$[l,m,n]\in\mathcal L$.
These $q+1$ lines meet the exterior line $\ell$ of $\pi_0$ in a $\pi_0$-\psline\, which, under the collineation $\sigma$, maps to a $\Bpi$-\psline\  $d$ of $\li$. By Theorem~\ref{sublinesinBB}, in $\PG(6,q)$, $d$ corresponds to the $2$-regulus $[d]$ which we denote by $\RR$, so 
$\RR=[d]=\{[W_\varepsilon]=[S_{\varepsilon^{1-q}}] \st \varepsilon\in\mathscr{W}\}$,
where ${\mathscr W}=\{\varepsilon=\varepsilon_{l,m,n}=l+m\tau+n\tau^2\st  
(l,m,n)\in\mathcal L\}$.
For each $\alpha\in\gfqcs$, consider the set of points $t_\alpha=\{ ([\varepsilon\alpha],[\varepsilon^q\alpha],0)\st\varepsilon\in\mathscr{W}\}$. As $\mathscr W$ is closed under addition, $t_\alpha$  is a line of $\si\cong\PG(5,q)$, further $t_\alpha$  meets every plane in $\RR$. Hence $t_\alpha$  is a ruling line of the $2$-regulus $\RR$. 

%

By Theorem~\ref{psline-converse}(2), the $2$-regulus $\RR$ meets a cover plane of the conic cover $\EC$ in a conic $\C_k=[C_k]\cap\RR$ for $k\in\KK$. To show that the conic $\C_k$ is $\EC$-special, we need to extend it to $\PG(5,q^3)$, and show that it meets the three transversal lines of $\EC$. To do this, we extend the $2$-regulus $\RR$ of $\si\cong\PG(5,q)$ to a $2$-regulus $\RR^*$ of $\PG(5,q^3)$, 
so $\C_k^*=[C_k]^*\cap\RR^*$. 
We then use coordinates to show that one of the planes of $\RR^*$ contains the transversal line $\gc^{q^2}$ of $\EC$, and then deduce that $\C_k^*$ meets   $\gc^{q^2}$.

To extend $\RR$ to a $2$-regulus $\RR^*$ of $\PG(5,q^3)$, we  find four lines in $\PG(5,q^3)$ that meet each extended plane of $\RR$. As 
 a $2$-regulus is uniquely determined by four ruling lines in general position, we can use these four lines to define the $2$-regulus $\RR^*$. The transversal line $\gs$ of the regular $2$-spread $\S$ can be used as one of our ruling lines, for the other three ruling lines, we use the extended lines
 $t_1^*$, $t_\tau^*,t_{\tau^2}^*$, which each meet every plane of $\RR$. 
 So $\RR^*$
 is the $2$-regulus of $\PG(5,q^3)$ determined by the four ruling lines $t_1^*$, $t_\tau^*,t_{\tau^2}^*,\gs$ (which are in general position), and further  
$\RR^* \cap \si=\RR$.

We now exhibit a plane $\gamma$ of $\RR^*$ that contains the transversal line $\gc^{q^2}$ of the conic cover $\EC$.
Extend the set $\mathcal L$ to ${\mathcal L}^*=\{(l,m,n)\st l,m,n\in\gfqc {\rm \ not\ all\ zero,\  and\ }  la+mb+nc=0\}$. 
We use the matrix $M_\tau$ defined in Section~\ref{sec:eg:Mk}, and write $M=M_\tau$. 
The ruling line $t^*_{\tau^i}$, $i=0,1,2$, has points $P_{\tau^i,l,m,n}$ with $(l,m,n)\in\mathcal L^*$, where
$P_{\tau^i,l,m,n}=l(\MMtau^i[1],\MMtau^i[1],0)+m(\MMtau^i[\tau],\MMtau^i[\tau^q],0)+n(\MMtau^i[\tau^2],\MMtau^i[\tau^{2q}],0)$.
Recall that the \orsl\ $d$ corresponds to the fixed point $P=(a,b,c)\in\pi_0$. 
Consider the following $(l,m,n)\in\L^*$,
\begin{eqnarray}\label{lmnabc}
l=c\tau-b\tau^2,\ \ \ m=a\tau^2-c,\ \ \ n=b-a\tau.
\end{eqnarray}
Note that for these $l,m,n$ we have
\begin{eqnarray}\label{lmn-tau-eqn}
l+m\tau+n\tau^2=0.
\end{eqnarray}
For  $l,m,n$ as in (\ref{lmnabc}),
consider the plane $\gamma$ spanned by the three points $P_{1,l,m,n}\in t_1^*, P_{\tau,l,m,n}\in t_\tau^*,P_{\tau^2,l,m,n}\in t_{\tau^2}^*$. 
We first show that $\gamma$ is a plane of the $2$-regulus $\RR^*$ by showing that the fourth ruling line  $g_\ES$ of $\RR^*$ also meets $\gamma$. By Theorem~\ref{gs-gt-gc}, $\gs=\langle A_1,A_2\rangle$, we show that  $\gs$ meets $\gamma$ by showing that  the point $A_2$  lies in 
$\gamma$. 
With $l,m,n$ given by (\ref{lmnabc}), consider the following point $F=p_0P_{1,l,m,n}+ p_1P_{\tau,l,m,n}+p_2P_{\tau^2,l,m,n}$ of $\gamma$. To simplify the notation, we use the  the point $A=(p_0,p_1,p_2)^t$, and matrix $U_0=p_0I+p_1\MMtau+p_1\MMtau^2$  defined in Section~\ref{sec:eg:Mk}, and note that $U_0[\alpha]=\alpha A$. We have
\begin{eqnarray*}
F
&=&( lU_0[1]+m U_0[\tau]+nU_0[\tau^2],\ lU_0[1]+mU_0[\tau^q]+nU_0[\tau^{2q}],\ 0)\\
&=&(lA+m\tau A+n\tau^2 A,\ lA+m\tau^q A+n\tau^{2q} A ,\ 0).
\end{eqnarray*}
By (\ref{lmn-tau-eqn}),  $F\equiv([0],A,0)=A_2$, and by Lemma~\ref{S-coord-g}, $\gs=\langle A_1,A_2\rangle$, so $F\in\gs\cap\gamma$. That is, the four ruling lines $t_1^*,t_\tau^*,t_{\tau^2}^*,\gs$ of the $2$-regulus $\RR^*$ all meet the plane $\gamma$,  and so $\gamma$ is a plane of $\RR^*$.

We now show that the transversal line 
 $g_\EC^{q^2}$ of $\EC$ lies in the plane 
 $\gamma$ of $\RR^*$.
 Let $G=p_0^{q^2}P_{1,l,m,n}+
p_1^{q^2}P_{\tau,l,m,n}+
p_2^{q^2}P_{\tau^2,l,m,n}$, and note that $G\in\gamma$. We use the matrix $U_2=p_0^{q^2}I+p_1^{q^2}\MMtau+p_1^{q^2}\MMtau^2$  defined in Section~\ref{sec:eg:Mk}, and note that $U_2[\alpha]=\alpha^{q^2} A^{q^2}$, so we 
have
\begin{eqnarray*}
G
&=&(lU_2[1]+m U_2[\tau]+n^2U_2[\tau^2],\ lU_2[1]+mU_2[\tau^q]+nU_2[\tau^{2q}],\ 0)\\
&=&(lA^{q^2}+m\tau^{q^2} A^{q^2}+n\tau^{2q^2} A^{q^2},\ lA^{q^2}+m\tau A^{q^2}+n\tau^{2} A^{q^2} ,\ 0).
\end{eqnarray*}
By (\ref{lmn-tau-eqn}), $G\equiv(A^{q^2},[0],0)=A_1^{q^2}$, so
 $\gamma$ contains the points $G=A_1^{q^2}$ and $F=A_2$. Hence by Theorem~\ref{gs-gt-gc}, $\gamma$ contains the transversal line $g_\EC^{q^2}=\langle  A_1^{q^2},A_2\rangle$ of $\EC$.

We showed above that the $2$-regulus $[d]=\RR$ meets a cover plane $[C_i]$ of $\EC$ in a conic $\C_i$. We want to show that $\C_i$ is a $\EC$-special conic, that is, we want to show that in $\PG(6,q^3)$, the extended conic $\C_i^*=[C_i]^*\cap\RR^*$ contains the three points $\gc\cap[C_i]^*,
\gc^q\cap[C_i]^*,
\gc^{q^2}\cap[C_i]^*$. We have shown that the transversal line $\gc^{q^2}$ of $\EC$ lies in a plane $\gamma$ of $\RR^*$. As the extended cover plane $[C_i]^*$ meets the transversal line $\gc^{q^2}$ in a unique point denoted $P_i$,  we have $P_i=[C_i]^*\cap\gc^{q^2}=[C_i]^*\cap\gamma\in[C_i]^*\cap\RR^*=\C_i^*$. Hence $\C_i^*$ contains the point $\gc^{q^2}\cap[C_i]^*$,
and hence it also contains the conjugate points $\gc^q\cap[C_i]^*,
\gc\cap[C_i]^*$. That is, the conic $\C_i=[C_i]\cap\RR$ is a $\EC$-special conic, completing the proof of part 1.

%

We now tackle part 2.  Let $\beta$ be the collineation defined in (\ref{def:beta}) acting on $\si^*\cong\PG(5,q^3)$.
Consider the line  $t_1$ defined above which lies in the tangent cover plane $[T_1]$. By the proof of Lemma~\ref{coord-covers-new}, $\beta$ maps the cover plane $[T_1]\in\ET$ to the cover plane $[C_1]\in\EC$. Hence the line
 $u_1=\beta(t_1)$ lies in the cover plane $[C_1]$ of the conic cover $\EC$. As $u_1$ meets $q+1$ planes of $\ES$, these planes form a $2$-regulus of $\ES$ which we denote $\R_1$.
Further, by Lemma~\ref{theta-action}, all the conic cover planes meet $\R_1$ in ruling lines of $\R_1$. 
By Theorem~\ref{sublinesinBB}, the $2$-regulus $\R_1$ corresponds to an \orsl\ $d'$ of $\li$ in $\PG(2,q^3)$. If $d'$ were a $\Bpi$-\psline\ of $\ES$, then by Theorem~\ref{psline-special}(1),  the tangent cover planes would meet the $2$-regulus $[d']=\R_1$ in ruling lines of $\R_1$.  However   we have just shown that the conic cover planes meet $[d']=\R_1$ in ruling lines of $\R_1$. Hence $d'$ is a $\Bpi$-\dcsline\ of $\ES$ in $\PG(2,q^3)$, and so we are in case 2.

As before, we want to extend the $2$-regulus $\R_1$ of $\si$ to a $2$-regulus $\R_1^*$ of $\si^*\cong\PG(5,q^3)$. 
We need to  determine coordinates of four ruling lines of the $2$-regulus $\R_1$. Let $\varepsilon\in\mathscr W$, then the point $([\varepsilon],[\varepsilon^{q^2}],0)$ lies in $u_1$. Further, it lies in the spread element $[S_{\varepsilon^{1-q^2}}]$. 
 Hence the $2$-regulus $\R_1$ consists of the planes $\{ [S_{\varepsilon^{1-q^2}}] \st \varepsilon \in {\mathscr W}\}$. As $[S_{\varepsilon^{1-q^2}}]$ has elements $([\varepsilon^{1-q^2}x],[x],0)\equiv([\varepsilon x],[\varepsilon^{q^2}x],0)$, the ruling lines of the $2$-regulus $\R_1$ are given by
$
v_\alpha=\{([\alpha \varepsilon],[\alpha \varepsilon^{q^2}],0)\st \varepsilon \in {\mathscr W}\},
$
for $\alpha\in\gfqcs$ (note that $u_1=v_1$). 
We extend the $2$-regulus $\R_1$ to the $2$-regulus $\R_1^*$
 uniquely determined by the four ruling lines 
$v_1^*=u_1^*,v_\tau^*,v_{\tau^2}^*,g_\ES$, where
$v_{\tau^i}^*=\{Q_{\tau^i,l,m,n}\st (l,m,n)\in \A^*\}$, and 
$Q_{\tau^i,l,m,n}=l(\MMtau^i[1],\MMtau^i[1],0)+m(\MMtau^i[\tau],\MMtau^i[\tau^{q^2}],0)+n(\MMtau^i[\tau^{2}],\MMtau^i[\tau^{2q^2}], 0)$, $i=0,1,2$.
 For $l,m,n$ as in (\ref{lmnabc}), consider the plane $\delta=\langle
Q_{1,l,m,n},
Q_{\tau,l,m,n},
Q_{\tau^2,l,m,n}
\rangle$.
Clearly $\delta$ meets the ruling lines $v_1^*,v_\tau^*,v_{\tau^2}^*$ of $\R_1^*$. We show that $\delta$ lies in the $2$-regulus $\R_1^*$ by showing that the fourth ruling line  $g_\ES$ of $\R_1^*$ also meets $\delta$. 
Consider the  point $H=p_0Q_{1,l,m,n}+p_1Q_{\tau,l,m,n}+p_2Q_{\tau^2,l,m,n}$ of $\delta$, then
\begin{eqnarray*}
H
&=&(lU_0[1]+mU_0[\tau]+nU_0[\tau^{2}],\ lU_0[1]+mU_0[\tau^{q^2}]+nU_0[\tau^{2q^2}],\ 0)\\
&=&((l+m\tau+n\tau^{2})A,\ (l+m\tau^{q^2}+n\tau^{2q^2})A,\ 0).
\end{eqnarray*}
By (\ref{lmn-tau-eqn}),  $H\equiv([0],A,0)=A_2$, hence $\gs=\langle A_1,A_2\rangle$ meets $\delta$ in the point $H=A_2$, and so the plane
 $\delta$ lies in the $2$-regulus $\R_1^*$.

We now show that the transversal line $g_\ET^q$ of the tangent cover $\ET$ lies in this plane $\delta$. Let $J=p_0^{q}Q_{1,l,m,n}+p_1^{q}Q_{\tau,l,m,n}+p_2^{q}Q_{\tau^2,l,m,n}$, and note that $J\in\delta$. Now 
\begin{eqnarray*}
J&=&(lU_1[1]+mU_1[\tau]+nU_1[\tau^{2}],\ lU_1[1]+mU_1[\tau^{q^2}]+nU_1[\tau^{2q^2}],\ 0)\\
&=&((l+m\tau^q+n\tau^{2q})A^{q},\ (l+m\tau+n\tau^{2})A^{q},\ 0).
\end{eqnarray*}
By (\ref{lmn-tau-eqn}),  $J\equiv(A^q,[0],0)=A_1^q$, thus 
$\delta$ contains the points $H=A_2$, $J=A_1^q$, and so by Theorem~\ref{gs-gt-gc}, $\delta$ contains the transversal line $g_\ET^q=
\langle A_2,A_1^q\rangle$ of $\ET$. 
Similar to the argument for part 1 above, a plane of $\ET$ meets the $2$-regulus $[d']=\R_1$ in a conic, which is $\ET$-special conic as $g_\ET^q\subset\delta\subset\R_1^*$. 
That is, a $\Bpi$-\dcsline\ of $\li$ corresponds to a $2$-regulus of $\si$ that meets tangent cover planes in  $\ET$-special conics, so part 2 is proved.
\end{proof}


\section{Conclusion}\Label{sec:concl}

An investigation into the interaction between an exterior \orsp\ $\pi$ of $\PG(2,q^3)$, and its exterior splash on $\li$ began in \cite{BJ-ext1}.  The main focus of \cite{BJ-ext1} was to show that exterior splashes are projectively equivalent to scattered linear sets of rank 3, covers of circle geometries, Sherk sets of size $q^2+q+1$. Further, we investigated the geometric relationship between the \orsls\ of $\ES$ and the points of $\pi$. 
The current article focusses on using the Bruck-Bose representation in $\PG(6,q)$ to continue the study of exterior splashes, in particular their interplay with \orsps. The notion of special conics and special twisted cubics is closely tied with this interplay. This theme is continued in \cite{BJ-ext3} where the authors show that $(\pi,\li)$-special conics of $\pi$ in $\PG(2,q^3)$ correspond to $\EC$-special conics in $\PG(6,q)$.

\end{document}